\documentclass[a4paper,12pt]{amsart}
\usepackage[utf8]{inputenc}
\usepackage{lmodern}
\usepackage{amsthm}
\usepackage{amssymb}
\usepackage{amsmath,amscd}
\usepackage{enumitem,mathrsfs}
\usepackage{tikz-cd}

\usepackage[top=3 cm, bottom=3 cm, left=2.7 cm, right=2.7cm]{geometry}

\title{A short note on Manin-Mumford}

\newtheorem{lem}{Lemma}

\newtheorem{corollary}{Corollary}
\newtheorem{prop}{Proposition}

\newcommand{\Q}{\mathbb{Q}}

\newcommand{\Z}{\mathbb{Z}}
\newcommand{\G}{\mathbb{G}}

\newcommand{\R}{\mathbb{R}}

\newcommand{\aaa}{\underline{a}}
\newcommand{\bbb}{\underline{b}}

\author{Harry Schmidt}
\address{harry.Schmidt@unibas.ch}
\begin{document}
\maketitle
\begin{abstract} We give a short proof of Manin-Mumford in the multiplicative group based on the pigeon-hole principle and the so-called structure theorem for anomalous subvarieties of $\G_m^n$. The arguments appear to be new and perhaps applicable in other situations. 
\end{abstract}
We start right away with some definitions. Let $\underline{a} \in \mathbb{Z}^n$ be  a non-zero integer vector and $N$ a rational integer $N \geq 2$. We set $S_{\underline{a}} = \Z \aaa + N\Z^n $ and $B_R = \{x \in \R^n; |x|_\infty \leq R \}$ for the max-norm $|\cdot|_\infty$. 
\begin{lem}\label{1} Let $N \geq 2^{2n}+1$ and $\underline{a} = (a_1, \dots, a_n) \in \Z^n$ be such that $\text{gcd}(a_1, \dots, a_n,N) = 1$.  Then $S_{\underline{a}} \cap B_{N^{1-\frac 1{2n}}}$ contains a non-zero vector. 
\end{lem}  
\begin{proof} As $\text{gcd}(a_1, \dots, a_n,N) = 1$ the reduction of $\underline{a}$ $\mod N$ has exact order $N$ in $\Z^n/N\Z^n$. Thus $S_{\aaa} \cap[0,N)^n$ contains $N$ distinct elements. We can cover the box $[0,N)^n$ by at most $([N/N^c] +1)^n \leq N^{n(1-c)}(1+N^c/N)^n$ boxes of side length $N^c$ and choosing $c = 1- 1/2n$ we find that we need at most $2^nN^{\frac 12}$ boxes of side-length $N^{1- \frac1{2n}}$.  For $N \geq 2^{2n}+1$ we have $2^nN^{\frac 12} < N$ and so there is a box containing at least two distinct elements of $S_{\underline{a}}$. Taking their difference we obtain the claim. 
\end{proof}

We need a second statement which for convenience reasons we state as a Lemma. 
\begin{lem}\label{2} Let $k \geq 2$ be a rational integer and $e = \text{gcd}(N,k)$. There exists an integer $l$ such that 
\begin{align*}
k = le \mod N \text{ and } \text{gcd}(l,N)= 1. 
\end{align*} 
\end{lem}
\begin{proof} 
We can set $l = k/e + fN/e$ where $f =\prod_{p|N, p\nmid k/e}p$ and check that it satisfies the requirements. 
\end{proof} 
Now comes the Lemma actually relevant for the proof of Manin-Mumford. 

\begin{lem}\label{3} Let $(\zeta_1, \dots, \zeta_n) \in \G_m^n$ be of order $N$. There exists an integer $e \leq N^{1-\frac1{2n}}$ dividing $N$ and a primitive $N$-th root of unity $\zeta_N$ such that $\zeta_i = \zeta^{(i)}_{e}\zeta_N^{k_i}, i = 1,\dots, N$ where $|k_i| \leq N^{1-\frac1{2n}}/e$ and $\zeta^{(i)}_{e}$ is an $e$-th root of unity for $i=1, \dots, n$. 
\end{lem}
\begin{proof} As $\underline{\zeta} = (\zeta_1, \dots, \zeta_n) $ is of order $N$ we can write it as $\underline{\zeta} = \zeta^{\aaa}$ for a primitive $N$-th root of unity $\zeta$ and an integer vector $\aaa = (a_1, \dots, a_n)$ such that $\text{gcd}(a_1, \dots, a_n,N) = 1$. By Lemma \ref{1} we can find a non-zero integer vector $\bbb$ and an integer $k$ such that $\bbb = k\aaa \mod N$ and $|\bbb|_\infty \leq N^{1 -\frac1{2n}}$. Let $e$ be the greates common divisor of $k$ and $N$. Then $e$ divides all entries of $\bbb$ and we can write  $\bbb = e\underline{c}$ for an integer vector $\underline{c}$ satifying $|\underline{c}|_\infty \leq N^{1-\frac1{2n}}/e$. By Lemma \ref{2} we obtain that $ef\underline{c} = e\aaa \mod N$ for an integer $f$ invertible modulo $N$. Thus setting $\zeta_N = \zeta^f$ we obtain the Lemma. 
\end{proof}

As a corollary of Lemma \ref{3} we obtain the following statement that brings us closer to Manin-Mumford. (In fact it already implies Manin-Mumford for curves in $\G_m^n$.)
\begin{corollary}\label{cor} Let $V \subset \G_m^n$ be an algebraic variety defined over a number field $K$. There exists an integer $M$ depending only on the degree of $K$ and the geometric degree of $V$ such that if $\underline{\zeta} \in V(\overline{\Q})$ is of order  at least $M$ then $\underline{\zeta}$ is contained in a torsion co-set $T$  (of positive dimension) that is contained in $V$.    
\end{corollary}
\begin{proof} Suppose $V$ is defined by the vanishing of $P_1,\dots, P_d$. Their total degree is bounded in terms of the degree of $V$ and we argue with all of them simultaneously. So let $P \in K[X_1,\dots, X_n]$ be one of the polynomials $P_1, \dots, P_d$ and let $\underline{\zeta}$ be a point of order $N$ on $V$. We may assume that the degree of $P$ is maximal among the degrees of $P_1,\dots, P_d$. Then $P(\underline{\zeta}) = 0$ and by Lemma \ref{3}  we can write 
$$(\zeta_1, \dots, \zeta_n) = (\zeta_e^{(1)}\zeta_N^{k_1}, \dots, \zeta_e^{(n)}\zeta_N^{k_n})$$ 
for $e$-th roots of unity $\zeta_e^{(i)}, i = 1, \dots, n$ and $|k_i| \ll N^{\frac 12}/e$. The Laurent-polynomial 
$$p (x) =P(\zeta_e^{(1)}x^{k_1}, \dots, \zeta_e^{(n)}x^{k_n}) \in K(\zeta_e)[x,x^{-1}], (\zeta_e \text{ a primitve $e$-th  root of } 1)$$
has degree at most $2\deg(P)N^{\frac12}/e$ and is defined over a field of degree at most $[K:\Q]e$ over $\Q$. However it vanishes at $x = \zeta_N$ which has degree at least $c_\epsilon N^{1- \epsilon}$ over $\Q$. Thus taking Galois conjugates we find that $p$ has at least $c_\epsilon N^{1-\epsilon}/([K:\Q]e)$ zeroes. So if $p$ is not identically zero  $c_\epsilon N^{1-\epsilon}/([K:\Q]e) \leq 2\deg(P)N^{1-\frac1{2n}}/e$ and  picking $\epsilon < \frac1{2n}$ we find that $N$ is bounded in terms of $[K:\Q]$ and $\deg(P)$ unless $p$ is the zero-polynomial. However if $p$ is the zero-polynomial then $\underline{\zeta}$ is contained in a torsion co-set lying in the zero-locus of $P$. As we performed this argument simultaneously for all $P_1, \dots, P_d$ we are finished.    
\end{proof}


We have proved that the set of torsion points on a variety $V$ is contained in a finite union of points and a union of torsion co-sets of dimension 1 contained in $V$. Somehow surprisingly this is enough to deduce the full Manin-Mumford conjecture using the structure theorem for anomalous subvarieties. 
\begin{prop} Let $V$ be a subvariety of $\G_m^n$. Suppose that 
\begin{align*}
V\cap (\G_m^n)_{\text{tors}} \subset  \bigcup_{\text{finite}  }\zeta \bigcup_{T \subset V }T, \ T \text{ torsion co-set of dimension 1 },  \zeta \text{ a root of 1.}
\end{align*}
Then 
\begin{align*}
\overline{V\cap (\G_m^n)_{\text{tors}}}^{\text{Zariski}} =   \bigcup_{\text{finite}  } H, \ H  \text{ torsion co-set}.
\end{align*}
\end{prop}
\begin{proof} We argue by induction on $\dim(V) + n$. For $\dim(V) + n \leq 2$ the statment is trivial.  We may assume that $V$ is irreducible by picking a component first. Each torsion co-set $T$ contained in $V$ is a $V$-anomalous subvariety of $V$. By the structure theorem \cite[Theorem 1.4 b)]{BMZstructure} there is a finite set of tori $\Phi$ such that the maximal $V$-anomalous variety $Y_T\subset V $ that contains such a torsion co-set $T$ is contained in a translate of a torus $H \in \Phi$. As $Y_T$ contains $T$ this translate is in fact a torsion translate.  If $Y_T = V$ for such a $Y_T$ then $V$ is contained in a torsion translates of a torus $H$ and thus by projecting a torsion translate of $V$ to  $H$ we can lower the dimension of $n$ and proceed by induction. Thus we may and will assume that for all  $Y_T\subset V$, $\dim(Y_T) < \dim(V)$. If the number of the maximal $Y_T$ is finite we can proceed by induction by restricting to $Y_T$ and thus lowering $\dim(V)$. Hence the remaining case to consider is that there is a fixed torus $H$ for which there are infinitely many $g \in (\G_m^n)_{\text{tors}}$ such that $\dim(gH \cap V) > \dim(V) + \dim H - n$. By translating $V$ by a torsion point if necessary we may assume that $\dim(H \cap V) > \dim(V) + \dim(H) -n$ and  we can read this as $\dim(V/H) < \dim (\G_m^n/H)$. Moreover a $g$ as in the previous sentence has the property that $[g] \in V/H$ where $g$ is the class of $g$ in $\G_m^n/H$. By induction  we deduce that the union of all $[g] \in V/H$ is a finite union $U$ of torsion translates (and points)  of $\G_m^n/H$ of positive co-dimension. We can write $U = H + U_1$ where $U_1$ is a representative of $U$ which is a finite union of torsion translates and points in $\G_m^n$ of positive co-dimension and we deduce that $Y_T \subset U + U_1 \neq \G_m^n$ for all $Y_T$ contained in a translate of $H$. By intersecting $V$ with $U + U_1$ we lower the dimension of $V$ and thus can proceed by induction.

\end{proof}
We finish by pointing out some of the differences between the present arguments and the proof of Tate presented in \cite[Theorem 6.1]{LangFundamentals}. Tate's argument also uses the fact that Galois orbits of roots of unity are large. But his strategy only requires a weaker polynomial growth in $N$, while our approach exploits the fact that the degree of a primitive $N$-th root of unity $\zeta$ grows like $\gg_\epsilon N^{1-\epsilon}$ (for all $\epsilon > 0$). Another important ingredient in Tate's proof is Bézout's theorem. More precisely one needs to estimate the number of points in the intersection $lC \cap C$ where $C$ is a curve and $lC$ is the curve obtained by multiplying the coordinate functions of $C$ by $l$ (or raising them to the power $l$) which implicitly uses the fact that we can estimate the degree of $lC$ appropriately. Our approach is more direct, in that we only need a naive estimate on the number of zeroes of a polynomial. The use of intersection theory is also central to Hindry's work \cite{Hindryautour} and it is somewhat interesting that one can replace it with the structure theorem of Bombieri, Masser and Zannier in our situation. Particularly since it can be proven without using intersection theory (for example by using functional transcendence and $o$-minimality). It is conceivable that the current proof can be extended to situations where we have a Tate parametrization of our group variety. More interestingly, we can apply these arguments to polynomial dynamical systems where the parametrization by $\G_m$ is given by Böttcher coordinates.  This is the content of future work. Finally, if one works over function fields (of any characteristic) these arguments seem to apply to the $j$-line. This will also be exploited in later work.  
\subsection*{Acknowledgements:} I heartily thank Fabrizio Barroero, Gabriel Dill, Richard Griffon, Lars Kühne and Myrto Mawraki for helpful discussions during the development of this work and for reading a first draft. An additional thanks goes to Garbiel Dill for pointing out a mistake in an earlier draft.  

\bibliography{Galoisbounds}
\bibliographystyle{amsalpha}

\end{document}